\documentclass[12pt]{amsart}

\usepackage{amsfonts,amssymb,stmaryrd,amscd,amsmath,latexsym,amsbsy,graphicx,hyperref}

\newtheorem{theorem}{Theorem}[section]
\newtheorem{lemma}[theorem]{Lemma}
\newtheorem{proposition}[theorem]{Proposition}

\theoremstyle{definition}
\newtheorem{definition}[theorem]{Definition}

\newtheorem{example}[theorem]{Example}

\newcommand{\ben}{\begin{enumerate}}
\newcommand{\een}{\end{enumerate}}

\theoremstyle{plain}

\newtheorem*{sol}{Solution}

\theoremstyle{definition}

\theoremstyle{remark}

\newcommand{\solu}[1]{\begin{sol}{\bf (\ref{#1})}}

\begin{document} 

{\large\centerline{\bf Angle-restricted sets} 
\centerline{\bf and zero-free regions for the permanent}}

\vskip .1in

\centerline{Pavel Etingof (MIT)} 

\vskip .1in

\centerline{\bf To the memory of E. V. Glivenko}

\section{Introduction}

A subset $S\subset \Bbb C^*:=\Bbb C\setminus \lbrace{0\rbrace}$ is called a {\it zero-free region for the permanent} if the permanent of a square matrix (of any size $n$) with entries in $S$ is necessarily nonzero. The motivation for studying such regions comes from the work of A. Barvinok (\cite{B}), where he shows that the logarithm of the permanent of such a matrix can be computed within error $\varepsilon$ in quasi-polynomial time $n^{O(\log n-\log \varepsilon)}$ (while the problem of efficient computation of general permanents is hopelessly hard). Namely, it is shown in \cite{B} that the disk $|z-1|\le 1/2$ and a certain family of rectangles are zero-free regions, which enables efficient approximate computation of permanents of matrices with entries from these regions.\footnote{We note that a (randomized) efficient algorithm for computing the permanent of a matrix with nonnegative entries was proposed earlier in \cite{JSV}.} 

The goal of this note is to give a systematic method of constructing zero-free regions for the permanent. We do so by refining the approach of \cite{B} using the clever observation that a certain restriction on a set $S$ involving angles implies zero-freeness (\cite{B}); we call sets satisfying this requirement {\it angle-restricted}. This allows us to reduce the question to a low-dimensional geometry problem (notably, independent of the size of the matrix!), which can then be solved more or less explicitly. We give a number of examples, improving some results of \cite{B}. This technique also applies to more general problems of a similar kind, discussed in \cite{B2}. 

{\bf Acknowledgements.} This paper was inspired by the Simons lectures of A. Barvinok at MIT in April 2019; namely, it is a (partial) solution of a ``homework problem" given in one of these lectures. I am very grateful to A. Barvinok for useful discussions, suggestions and encouragement. I am also very indebted to two anonymous referees for thorough reading of the paper and very useful comments and corrections. 

I dedicate this paper to the memory of my teacher Elena Valerievna Glivenko, Professor of Applied Mathematics at the Moscow Oil and Gas Institute, where I was a student in the late 1980s. Her teaching and care made an enormous difference for all of us. 

\section{Definition and basic properties of angle-restricted sets}

For $u,v\in \Bbb C^*$ let $\alpha(u,v)\in [0,\pi]$ be the angle between $u$ and $v$. 
Let $\theta,\phi\in (0,2\pi/3)$. Note that if $u_1,...,u_n\in \Bbb C^*$ 
are such that $\alpha(u_i,u_j)\le \theta$ then there exists $\lambda\in \Bbb C^*$ 
such that $|{\rm arg}(\lambda u_i)|\le \theta/2$ for all $i$ (where we agree that ${\rm arg}(z)$ takes values in $(-\pi,\pi]$). 

\begin{definition}\label{sbp} (i) We denote by $A_{\theta,\phi}$ the set of subsets $S\subset \Bbb C^*$ such that for any $u_1,...,u_n\in \Bbb C^*$ with $\alpha(u_i,u_j)\le\theta$ for all $i,j$ and any 
$a_1,...,a_n,b_1,...,b_n\in S$, the numbers $v=\sum_i a_iu_i$ and $w=\sum_i b_iu_i$ are nonzero and $\alpha(v,w)\le\phi$. In other words, if $u_i$ belong to the angle $|{\rm arg}(z)|\le \theta/2$ then there exists $\mu\in \Bbb C^*$ such that $\mu v,\mu w$ belong to the angle $|{\rm arg}(u)|\le \phi/2$. 
We say that a set $S\subset \Bbb C^*$ is $(\theta,\phi)$-{\it angle restricted} if  $S\in A_{\theta,\phi}$. If $\theta=\phi$ then we denote $A_{\theta,\phi}$ by $A_\theta$. 

(ii)  We denote by $A_{\theta,\phi}^2$ the set of subsets $S\subset \Bbb C^*$ such that for any $a,b,c,d\in S$ the map $z\mapsto \frac{az+b}{cz+d}$ maps the angle $\lbrace{z\in \Bbb C^*:|{\rm arg}(z)|\le\theta\rbrace}$ into the angle $\lbrace{u\in \Bbb C^*:|{\rm arg}(u)|\le\phi\rbrace}$. In other words, $S\in A_{\theta,\phi}^2$ if and only if any $a,b,c,d\in S$ satisfy the condition of (i) for $n=2$. We denote $A_{\theta,\theta}^2$ by $A_\theta^2$. 

(iii) We denote by $B_{\theta,\phi}^2$ the set of subsets $S\subset \Bbb C^*$ such that 
for any $a,b\in S$ the map $z\mapsto \frac{az+b}{z+1}$ maps the angle $\lbrace{z\in \Bbb C^*:|{\rm arg}(z)|\le\theta\rbrace}$ into the angle $\lbrace{u\in \Bbb C^*: |{\rm arg}(u)|\le\phi/2\rbrace}$. We denote $B_{\theta,\theta}^2$ by $B_\theta^2$.   
\end{definition} 

{\bf Remark.}
Condition (i) for $n=2$ says that for any $u_1,u_2\in \Bbb C^*$ with $\alpha(u_i,u_j)\le\theta$ and $a,b,c,d\in S$ we have $\alpha(au_1+bu_2,cu_1+du_2)\le \phi$. This can be written as $|{\rm arg}(\frac{az+b}{cz+d})|\le \phi$, where $z:=u_1/u_2$, which implies that the two definitions of $A_{\theta,\phi}^2$ in (ii) are equivalent.  

\vskip .05in

It is clear that $A_{\theta,\phi}\subset A_{\theta,\phi}^2$ and $B_{\theta,\phi}^2\subset A_{\theta,\phi}^2$ (as $\frac{az+b}{cz+d}=\frac{az+b}{z+1}\cdot \frac{z+1}{cz+d}$), and that $A_{\theta,\phi}, A_{\theta,\phi}^2$ are invariant under rescaling by a nonzero complex number, while $B_{\theta,\phi}^2$ is invariant under rescaling by a positive real number. Also it is obvious that if $S$ belongs to any of these sets then so do all subsets of $S$. Finally, it is clear that any ray emanating from $0$ is in $A_\theta$, so we will mostly be interested in sets $S$ that are not contained in a line. 

The motivation for studying these notions comes from the following result of A. Barvinok (\cite{B}). 

\begin{theorem} (i) If $S\in A_{\pi/2}$ then any square matrix with entries from $S$ has nonzero permanent. 

(ii) The disk $|z-1|\le 1/2$ is in $A_{\pi/2}$. 
\end{theorem} 

This implies that any square matrix with entries $a_{ij}$ such that \linebreak $|a_{ij}-1|\le 1/2$ has nonzero permanent. This allowed A. Barvinok to give in \cite{B} an algorithm for efficient approximate computation of (logarithms of) permanents of such matrices with good precision. 

The sets $A_{\theta,\phi}$ for more general $\theta$ and $\phi$, also studied by A. Barvinok, have similar properties and applications (see \cite{B,B2}). Namely, as explained in \cite{B2}, the condition that $S\in A_{\theta,\phi}$ for suitable $\theta$ and $\phi$ guarantees that some quite general combinatorially defined multivariate polynomials $P(z_1, . . . , z_n)$, such as the graph homomorphism partition function, are necessarily non-zero whenever $z_1,...,z_n \in U$, and can be efficiently approximated there. 

The sets $A_{\theta,\phi}^2$, $B_{\theta,\phi}^2$ introduced here play an auxiliary role, but they are fairly easy to study (as their definition involves a small number of parameters), and yet we will show that a convex set belonging to $A_{\theta,\phi}^2$ must belong to $A_{\theta,\phi}$. 

\begin{proposition} \label{p2} (i) If $S\in A_{\theta,\phi}^2$ and $a,b\in S$ then $\alpha(a,b)<\pi-\theta$ and $\alpha(a,b)\le\phi$. 

(ii) If $S\in A_{\theta,\phi}^2$ and $a_1,...,a_n\in S$ then for any $u_1,...,u_n\in \Bbb C^*$ 
with $\alpha(u_i,u_j)\le \theta$ for all $i,j$ we have $\sum_j a_ju_j\ne 0$. 
\end{proposition} 

\begin{proof} 
(i) If $a,b\in S$ then $au_1+bu_2$ does not vanish if $\alpha(u_1,u_2)\le\theta$. 
Suppose $b/a=re^{i\psi}$ where $0\le \psi\le \pi$ (this can always be achieved by switching $a,b$ if needed). 
Then $\psi<\pi-\theta$, since otherwise we may take $u_2=1$, $u_1=-b/a$ (so that $\alpha(u_1,u_2)\le\theta$) and $au_1+bu_2=0$, a contradiction. Also $\psi\le \phi$, since otherwise $\alpha(au_1+bu_2,a(u_1+u_2))$  
for $u_1=1$ and $u_2=N\gg 1$ will exceed $\phi$. 

(ii) By (i) we have $\alpha(a_i,a_j)<\pi-\theta$ and $\alpha(a_i,a_j)\le \phi<2\pi/3$. 
Thus after rescaling by a complex scalar we may assume that 
$$
|{\rm arg}(a_j)|< \frac{1}{2}(\pi-\theta)
$$ 
for all $j$. Let $u_1,...,u_n\in \Bbb C^*$ with pairwise angles $\le\theta$. By rescaling by a complex scalar we may make sure that $|{\rm arg}(u_j)|\le\theta/2$. Then $|{\rm arg}(a_ju_j)|<\pi/2$, so ${\rm Re}(a_ju_j)>0$ for all $j$. Thus $\sum_j a_ju_j\ne 0$. 
\end{proof} 

\begin{proposition}\label{p2a} Let $\phi\le \pi/2$. Then 
a set $S\subset \Bbb C^*$ is in $A_{\theta,\phi}^2$ if and only if for all $a,b,c,d\in S$ the map $z\mapsto \frac{az+b}{cz+d}$ maps the angle $\lbrace{z\in \Bbb C^*:|{\rm arg}(z)|\le\theta\rbrace}$ into $\lbrace{u\in \Bbb C^*:|{\rm arg}(u)|\le\phi\rbrace}\cup \lbrace{0,\infty\rbrace}$.
\end{proposition}  

\begin{proof} Only the ``if" direction requires proof. 
It suffices to show that for $a,b\in S$ and $z\in \Bbb C^*$ with 
$|{\rm arg}(z)|\le \theta$ one has $az+b\ne 0$. Assume the contrary. 
For any $c\in S$, the map $w\mapsto \frac{aw+b}{cw+c}$ must map 
the angle $|{\rm arg}(z)|\le \theta$ to the set $\lbrace{u\in \Bbb C^*:|{\rm arg}(u)|\le\phi\rbrace}\cup \lbrace{0,\infty\rbrace}$, while mapping $z$ to $0$. Considering these maps for $c=a,b$ near $w=z$ (with $w/z>0$) and using that $\phi\le \pi/2$, we get that $b/a>0$, i.e., $z<0$, a contradiction.  
\end{proof} 

\section{Convexity and reduction to $n=2$}

The following theorem reduces checking that a convex set is $(\theta,\phi)$-angle restricted to checking that it is in $A_{\theta,\phi}^2$, which is just a low-dimensional geometry problem.  

\begin{theorem} (i) If $S\in A_{\theta,\phi}$ then so is the convex hull of $S$.    

(ii) If $S\in A_{\theta,\phi}^2$ is convex then $S\in A_{\theta,\phi}$. 
\end{theorem} 

\begin{proof}
(i) Let $CH(S)$ be the convex hull of $S$. Assume $S\in A_{\theta,\phi}$. Let $a_1,...,a_n,b_1,...,b_n\in CH(S)$. 
Then $a_i=\sum_j r_{ij}a_{ij}$ where $a_{ij}\in S$, \linebreak $r_{ij}> 0$ and $\sum_j r_{ij}=1$. 
Similarly, $b_i=\sum_k s_{ik}b_{ik}$ where $b_{ik}\in S$, $s_{ik}> 0$ and $\sum_k s_{ik}=1$. 
Let $u_1,...,u_n\in \Bbb C^*$ with angle between each two $\le\theta$. Let $u_{ijk}=r_{ij}s_{ik}u_i$. 
Consider 
$$
v:=\sum_{i,j,k}a_{ij}u_{ijk}=\sum_{i,j,k}a_{ij}r_{ij}s_{ik}u_i=\sum_{i,k}a_is_{ik}u_i=\sum_i a_iu_i
$$ 
and 
$$
w:=\sum_{i,j,k}b_{ik}u_{ijk}=\sum_{i,j,k}b_{ik}r_{ij}s_{ik}u_i=\sum_{i,j}b_ir_{ij}u_i=
\sum_i b_iu_i. 
$$
Since $a_{ij},b_{ik}\in S$, we have that $v,w\ne 0$ and the angle between them does not exceed 
$\phi$. Thus $CH(S)\in A_{\theta,\phi}$. 

(ii) Denote by $R_{n,\theta}\subset \Bbb C\Bbb P^{n-1}$ the set of points $\bold u=(u_1,...,u_n)$ such that the pairwise angles between $u_i$ and $u_j$ (when both are nonzero) are at most $\theta$. It is clear that $R_{n,\theta}$ is closed (hence compact). By Proposition \ref{p2}(ii) for any $a_1,...,a_n\in S$ we have $\sum_j a_ju_j\ne 0$. Now fix $a_1,...,a_n,b_1,...,b_n\in S$ and consider the function 
$$
f(u_1,...,u_n)={\rm Im}\log \frac{\sum_j a_ju_j}{\sum_j b_ju_j} 
$$
(we choose a single-valued branch of this function). The function $f$ is harmonic on $R_{n,\theta}$ in each variable. Let $\bold u\in R_{n,\theta}$ be a global maximum or minimum point of $f$. By the maximum principle\footnote{Note that using the coordinates $v_i:= \frac{u_i}{\sum_{j=1}^n u_j}$, $1\le i\le n-1$, we may identify $R_{n,\theta}$ with a closed region in $\Bbb C^{n-1}$. Thus we may apply the maximum principle for harmonic functions on subsets of a Euclidean space.},  we may choose $\bold u=(u_1,...,u_n)$ so that each $u_i$ is zero or has argument $\pm \theta/2$. By reducing $n$ if needed and relabeling, we may assume that all $u_j$ are nonzero and that $u_j=r_je^{i\theta/2}$ for $j=1,...,m$ and $u_j=r_je^{-i\theta/2}$ for $j=m+1,...,n$, where $r_j>0$ for all $j$. By rescaling by a positive real number, we may assume that $\sum_{j=1}^m r_j=r$ and $\sum_{j=m+1}^n r_j=1$. Thus we have 
$$
v=\sum_ja_ju_j=are^{i\theta/2}+be^{-i\theta/2},\quad w=\sum_jb_ju_j=cre^{i\theta/2}+de^{-i\theta/2},
$$
where 
$$
a=\sum_{j=1}^m a_jr_j/r,\ b=\sum_{j=m+1}^n a_jr_j,\ c=\sum_{j=1}^m b_jr_j/r,\ d=\sum_{j=m+1}^n b_jr_j.
$$
Since $S$ is convex and $a,b,c,d$ are convex linear combinations of the numbers $\lbrace{a_j,j\le m\rbrace}$, $\lbrace{a_j,j>m\rbrace}$, $\lbrace{b_j,j\le m\rbrace}$, $\lbrace{b_j,j>m\rbrace}$ respectively, we get that $a,b,c,d\in S$. Thus, using that $S\in A_{\theta,\phi}^2$ and setting $z=re^{i\theta}$, we see that the angle between $v$ and $w$ does not exceed $\phi$, as claimed. 
\end{proof} 

\begin{lemma}\label{l3} Let $S\in A_{\theta,\pi/2}^2$, and $a,b\in S$ with $b/a=x+iy$, $x,y\in \Bbb R$. 
Then we have $x\ge 0$ and
\begin{equation}\label{eq2}
|y|\le \frac{2\sqrt{x}+(x+1)\cos\theta}{\sin\theta}, 
\end{equation} 
and if $\theta>\pi/2$ then 
\begin{equation}\label{eq3} 
\left(x+\frac{1}{\cos\theta}\right)^2+y^2\le \tan^2\theta. 
\end{equation} 
In particular, if $\theta>\pi/2$ then 
$$
\frac{1-\sin\theta}{|\cos\theta|}\le x\le \frac{1+\sin\theta}{|\cos\theta|},
$$ 
i.e., $b/a$ is separated from the imaginary axis and from infinity (so any $S\in A_{\theta,\pi/2}^2$ is bounded). Moreover, conditions \eqref{eq2},\eqref{eq3}, together with condition \eqref{eq2} with $a$ and $b$ switched are also sufficient for the set $\lbrace{a,b\rbrace}$ to be in $A_{\theta,\pi/2}^2$. 
\end{lemma} 

\begin{proof} Let $a,b\in S$ with $b/a=x+iy$. Pick $u_1 = re^{\pm i\theta}$, $u_2 = 1$. The angle between $au_1 + bu_2$ and $au_1 + au_2$ does not exceed $\pi/2$. Hence the real part of $\frac{au_1+bu_2}{au_1+au_2}$ is non-negative. Thus, we have 
$$
{\rm Re}\left(\frac{re^{\pm i\theta}+x+iy}{re^{\pm i\theta}+1}\right)\ge 0,\ \forall r>0. 
$$
This yields 
$$
{\rm Re}\left((re^{\pm i\theta}+x+iy)(re^{\mp i\theta}+1)\right)\ge 0,\ \forall r>0, 
$$
i.e., 
$$
r^2+((x+1)\cos \theta\pm y\sin\theta)r+x\ge 0,\ \forall r>0. 
$$
This implies that $x\ge 0$, and minimizing with respect to $r$, we get
$$
(x+1)\cos\theta\pm y\sin\theta\ge -2\sqrt{x}, 
$$
which yields 
$$
|y|\le \frac{2\sqrt{x}+(x+1)\cos\theta}{\sin\theta},  
$$
as claimed. 

Similarly, the real part of $\frac{au_1+bu_2}{bu_1+au_2}$ is non-negative. Thus, we have
$$
{\rm Re}\left(\frac{re^{\pm i\theta}+x+iy}{(x+iy)re^{\pm i\theta}+1}\right)\ge 0,\ \forall r>0. 
$$
This yields 
$$
{\rm Re}((re^{\pm i\theta}+x+iy)((x-iy)re^{\mp i\theta}+1))\ge 0,\ \forall r>0, 
$$ 
i.e. 
$$
xr^2+(x^2+y^2+1)r\cos\theta+x\ge 0, \forall r>0. 
$$
This is satisfied automatically if $\theta\le \pi/2$, but if $\theta>\pi/2$ then minimizing the left hand side with respect to $r$ gives the condition 
$$
(x^2+y^2+1)\cos\theta+2x\ge 0,
$$
which is equivalent to \eqref{eq3}. 

Finally, to check that $\lbrace{a, b\rbrace} \in A_{\theta,\pi/2}^2$, it suffices to check that for any $u_1, u_2 \in \Bbb C^*$ that are within angle $\theta$ of each other, the angles 
$$
\alpha(au_1 +bu_2,au_1 +au_2),\ \alpha(au_1 +bu_2,bu_1 +au_2),\ \alpha(au_1+bu_2,bu_1+bu_2)
$$
 do not exceed $\pi/2$. These angles are harmonic functions 
 of $u_1/u_2$, so the maximum has to be attained on the boundary. Hence it suffices to choose $u_1 = re^{\pm i\theta}$ and $u_2 = 1$.  Thus, conditions \eqref{eq2},\eqref{eq3}, together with condition \eqref{eq2} with $a$ and $b$ switched are sufficient for the set $\lbrace{a,b\rbrace}$ to be in $A_{\theta,\pi/2}^2$, as claimed. 
 \end{proof} 

Thus we see that the region for $b/a$ is bounded by two parabolas given by \eqref{eq2} and their inversions under the circle $|z|=1$, as well as the circle given by \eqref{eq3} if $\theta>\pi/2$ (note that this circle is stable under inversion). 

\begin{proposition}\label{p3} Suppose that $\phi\le \pi/2$. Then 

(i) if $S\in A_{\theta,\phi}^2$ then the closure $\overline{S}$ of $S$ in $\Bbb C^*$ belongs to $A_{\theta,\phi}^2$; 

(ii) if $S\in A_{\theta,\phi}^2$ then the convex hull $CH(S)$ of $S$ belongs to $A_{\theta,\phi}^2$. 
\end{proposition}  

\begin{proof} (i) follows by continuity from Proposition \ref{p2a}, since the set 
$\lbrace{u\in \Bbb C^*:|{\rm arg}(u)|\le\phi\rbrace}\cup \lbrace{0,\infty\rbrace}$
is closed in the Riemann sphere. 

(ii) Let $a,b,b',c,d\in \Bbb C^*$ be such that the maps 
$z\mapsto \frac{az+b}{cz+d}$ and $z\mapsto \frac{az+b'}{cz+d}$ 
satisfy the condition of Proposition \ref{p2a}, $r\in [0,1]$ and \linebreak $b'':=rb+(1-r)b'$. 
We claim that the map $z\mapsto \frac{az+b''}{cz+d}$ 
also satisfies the condition of Proposition \ref{p2a}. It suffices to show this for $z\ne -d/c$. We have
$$
\frac{az+b''}{cz+d}=r\frac{az+b}{cz+d}+(1-r)\frac{az+b'}{cz+d},
$$
and $\frac{az+b}{cz+d},\frac{az+b'}{cz+d}$ belong to the set 
$\lbrace{u\in \Bbb C^*:|{\rm arg}(u)|\le\phi\rbrace}\cup \lbrace{0\rbrace}$, which is convex since $\phi\le \pi/2$. Hence $\frac{az+b''}{cz+d}$ also belongs to this set, as claimed. 

Also note that the condition of Proposition \ref{p2a} is invariant under the transpositions
$(a,b,c,d)\mapsto (b,a,d,c)$ and $(a,b,c,d)\mapsto (c,d,a,b)$, which generate a group $\Bbb Z_2\times \Bbb Z_2$ acting transitively on $a,b,c,d$. Now (ii) follows by using this symmetry 
and applying the above claim four times (to each of the four variables $a,b,c,d$). 
\end{proof} 

This proposition gives a simple method of constructing convex polygons which are in $A_{\theta,\pi}^2$ for $\phi\le \pi/2$ by doing a finite check on the vertices. We will see examples of this below.

\section{The sets $A_{\pi/2}^2$ and $B_{\pi/2}^2$}

From now on we focus on the case $\theta=\phi=\pi/2$ relevant for zero-free regions for the permanent. The general case can be treated by similar methods. 

\subsection{Explicit characterization} 
Let us give a more explicit characterization of the sets $A_\theta^2$ and $B_\theta^2$ for $\theta=\pi/2$. 
Let 
$$
F(a,b,c,d)=({\rm Im}(a\bar d-b\bar c))^2-4{\rm Re}(a\bar c){\rm Re}(b\bar d), 
$$
and 
$$
G_1(a,b)=(a_2-b_2)^2-4a_1b_1, \quad G_2(a,b)=(a_1-b_1)^2-4a_2b_2,
$$
where $a_1+ia_2=e^{i\pi/4}a, b_1+ib_2=e^{i\pi/4}b$, $a_j,b_j\in \Bbb R$. 
Note that 
$$
F(a,b,c,d)=F(b,a,d,c)=F(c,d,a,b)=F(d,c,b,a).
$$ 

\begin{lemma}\label{l1} (i) $S\in A_{\pi/2}^2$ if and only if for any $a,b,c,d\in S$ we have 
$F(a,b,c,d)\le 0$.

(ii) $S\in B_{\pi/2}^2$ if and only if $|{\rm arg}(a)|\le \pi/4$ for $a\in S$, and for any $a,b\in S$ we have $G_1(a,b)\le 0$, $G_2(a,b)\le 0$.
\end{lemma}

\begin{proof} (i) Suppose that $F(a,b,c,d)\le 0$ for all $a,b,c,d\in S$. 
Then ${\rm Re}(a\bar c)\ge 0$ for all $a,c\in S$ (as we can take $b=d$). Therefore, $\frac{az+b}{cz+d}\ne 0$ when ${\rm Re}(z)\ge 0$. Indeed, otherwise, we must have ${\rm Re}(b/a)=|a|^{-2}{\rm Re}(b\bar a)\le 0$, so ${\rm Re}(b/a)=0$ and $b/a=it$ for some real $t\ne 0$. But then $F(a,b,a,a)=t^2|a|^4>0$, a contradiction.  
 
Thus by the definition of $A^2_{\pi/2}$, it suffices to show that for $a,b,c,d\in S$ 
one has ${\rm Re}\frac{az+b}{cz+d}\ge 0$ whenever $z=it$, $t\in \Bbb R$. 
We have 
$$
\frac{ait+b}{cit+d}=\frac{(ait+b)(-\bar cit+\bar d)}{|cit+d|^2}
$$ 
and 
$$
{\rm Re}\left((ait+b)(-\bar cit+\bar d)\right)={\rm Re}(a\bar c)t^2-{\rm Im}(a\bar d-b\bar c)t+{\rm Re}(b\bar d). 
$$
Since ${\rm Re}(a\bar c),{\rm Re}(b\bar d)\ge 0$, the condition for this to be $\ge 0$ is that the discriminant of this quadratic function is $\le 0$, which gives the result. 

Conversely, if $S\in A_{\pi/2}^2$ then the above calculation shows that \linebreak $F(a,b,c,d)\le 0$ for all $a,b,c,d\in S$. 

(ii) Let $a'=e^{i\pi/4}a=a_1+ia_2,b'=e^{i\pi/4}b=b_1+ib_2$. The condition on $a',b'$ is that for $t\in \Bbb R$ we have 
${\rm Re}\frac{a'it+b'}{it+1}\ge 0$ and ${\rm Im}\frac{a'it+b'}{it+1}\ge 0$. We have 
$$
\frac{a'it+b'}{it+1}=\frac{(a'it+b')(-it+1)}{t^2+1},
$$
and 
$$
(a'it+b')(-it+1)=a't^2+(a'-b')it+b'=
$$
$$
=(a_1t^2-(a_2-b_2)t+b_1)+i(a_2t^2+(a_1-b_1)t+b_2).
$$
Since $a_1,a_2,b_1,b_2\ge 0$ (as seen by setting $t=0$ and $t=\infty$), the condition is that the discriminants of these two quadratic functions must be $\le 0$, which gives the result. 
\end{proof} 

\subsection{Examples}

\begin{example}\label{ex0} 
Lemma \ref{l1}(ii) implies that the interval $[a,b]\subset \Bbb R$ for $0<a\le b$ is in $B_{\pi/2}^2$ iff $b/a\le 3+2\sqrt{2}$. 
\end{example} 

\begin{example}\label{ex00}
Let $a=1/2$, $b=1+i/2$, $c=1-i/2$ and $d=3/2+t$. Let us find the largest $t>0$ for which 
$\lbrace{a,b,c,d\rbrace}$ is in $B_{\pi/2}^2$ (hence in $A_{\pi/2}^2$). Since $a,b,c$ belong to the disk $|z-1|\le 1/2$, which was shown by A. Barvinok in \cite{B} to belong to $B_{\pi/2}^2$, it suffices to check when $G_i(a,d)\le 0$, $G_i(b,d)\le 0$, $G_i(c,d)\le 0$. The first condition gives the inequality
of Example \ref{ex0}, which is $3+2t\le 3+2\sqrt{2}$, i.e. $t\le \sqrt{2}$. The second (or, equivalently,  third) condition gives the inequalities
$t^2\le 2t+3, (1+t)^2\le 3(2t+3)$ which hold for $0\le t\le \sqrt{2}$. Thus we find that the optimal value is $t=\sqrt{2}$ and the quadrilateral 
with vertices $1/2,1\pm i/2$ and $\frac{3}{2}+\sqrt{2}$:

\includegraphics[scale=0.4]{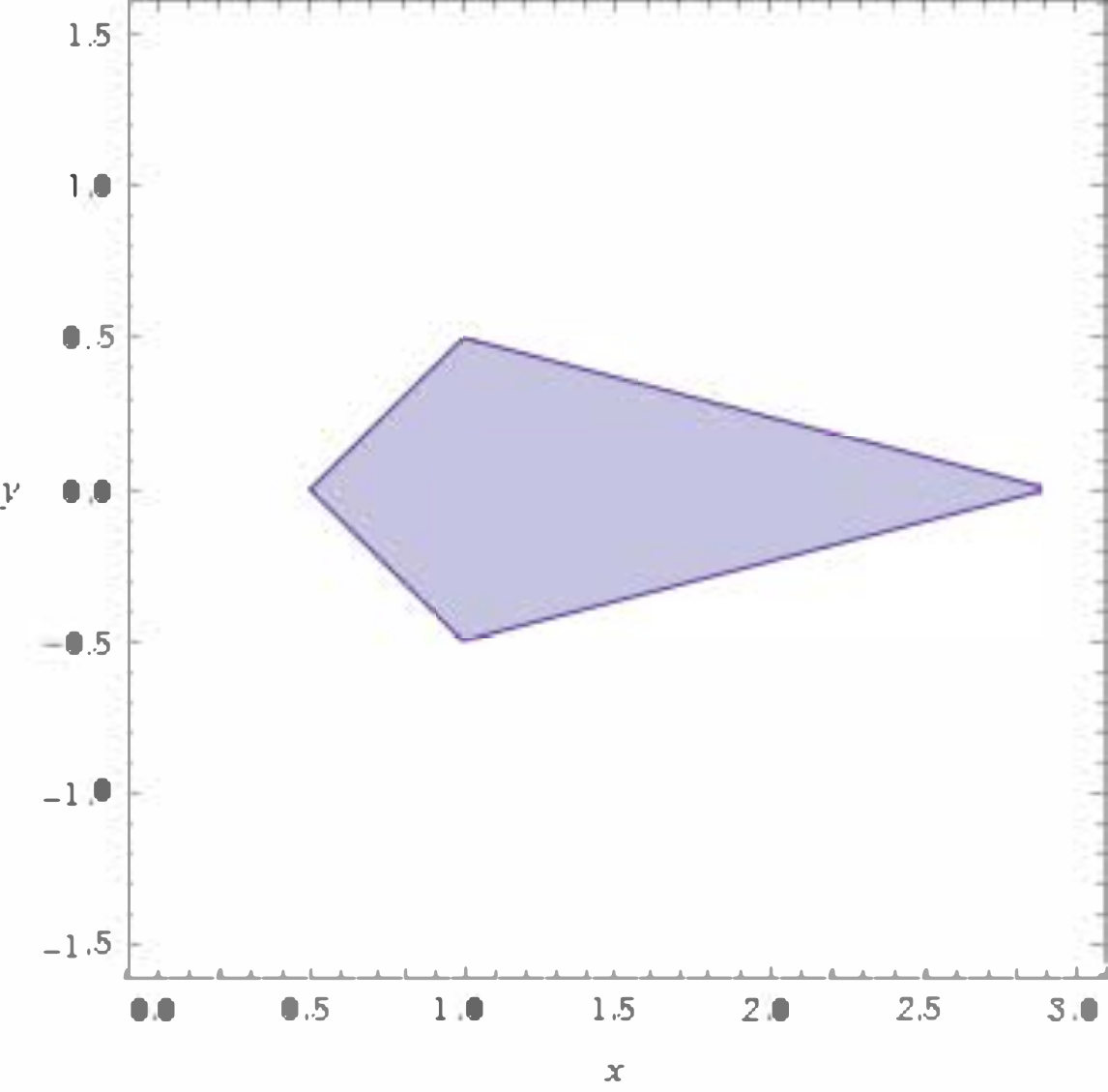}

\noindent is in $A_{\pi/2}^2$, hence in $A_{\pi/2}$ by Proposition \ref{p3}(ii); thus it is a zero-free region for the permanent. 
\end{example} 

\begin{example} Let us find the values of $t>1/2$ for which the union of the disk $|z-1|\le 1/2$ and the point 
$1+t$ belongs to $B_{\pi/2}^2$ (hence to $A_{\pi/2}^2$). 
Such $t$ are determined by the condition that $G_1(1+\frac{1}{2}e^{i(\phi-\pi/4)},1+t)\le 0$ 
for all $\phi$ (the condition involving $G_2$ is the same due to axial symmetry). This can be written as 
$$
(t+\frac{1}{\sqrt{2}}\cos \phi)^2\le 4(1+t)(1+\frac{1}{\sqrt{2}}\sin\phi)
$$
for all $\phi$. This gives 
$$
t\le 2+\sqrt{2} \sin\phi - \frac{\sqrt{2}}{2} \cos\phi + \sqrt{6\sqrt{2} \sin\phi - 2 \sqrt{2} \cos\phi - \sin 2\phi-\cos 2\phi + 9},
$$
and minimizing this function (numerically), we get the answer 
$$
t\le t_*=1.64.....
$$
Thus the ice cream cone, which is the convex hull 
of the disk $|z-1|\le 1/2$ and the point $1+t_*$  (significantly larger than the disk):

\includegraphics[scale=0.4]{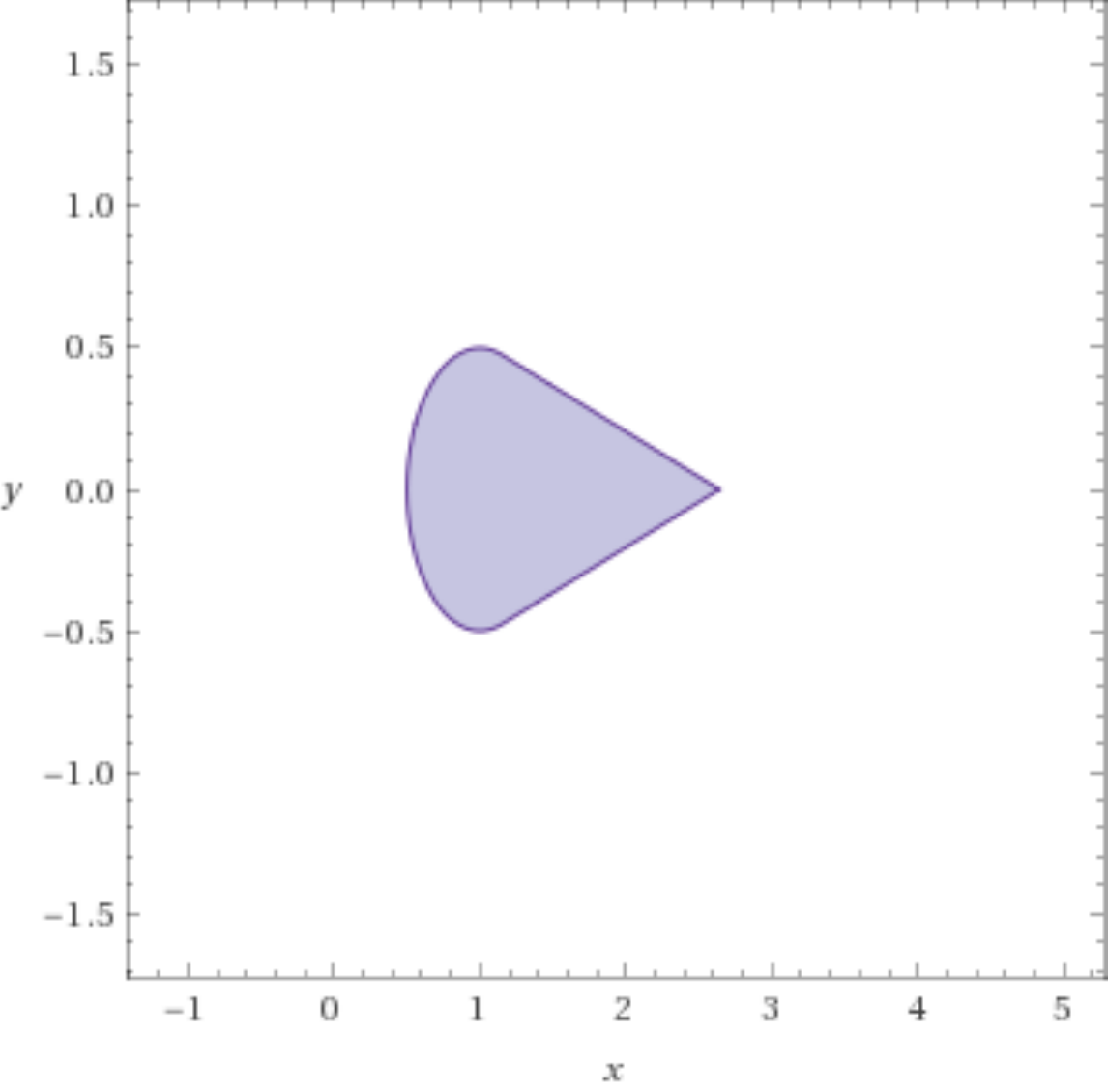}

\noindent belongs to $A_{\pi/2}^2$, hence to $A_{\pi/2}$ by Proposition \ref{p3}(ii), and thus is a zero-free region for the permanent. 
\end{example} 

\begin{example}\label{ex1} Let $S=\lbrace{a,b\rbrace}$, and $b/a=x+iy$. Let us compute when $S\in A_{\pi/2}^2$. By Lemma \ref{l3} 
the conditions for this are 
$$
y^2\le 4x,\quad y^2\le 4x(x^2+y^2).
$$
This gives 
\begin{equation}\label{3/2bd}
|y|\le 2\sqrt{x};\text{ and } |y|\le \frac{2x^{3/2}}{\sqrt{1-4x}},\ x<1/4.
\end{equation}
So we get a region which is bounded by a parabola and its inversion with respect to the circle $|z|=1$, which is a cissoid of Diocles: 

\includegraphics[scale=0.75]{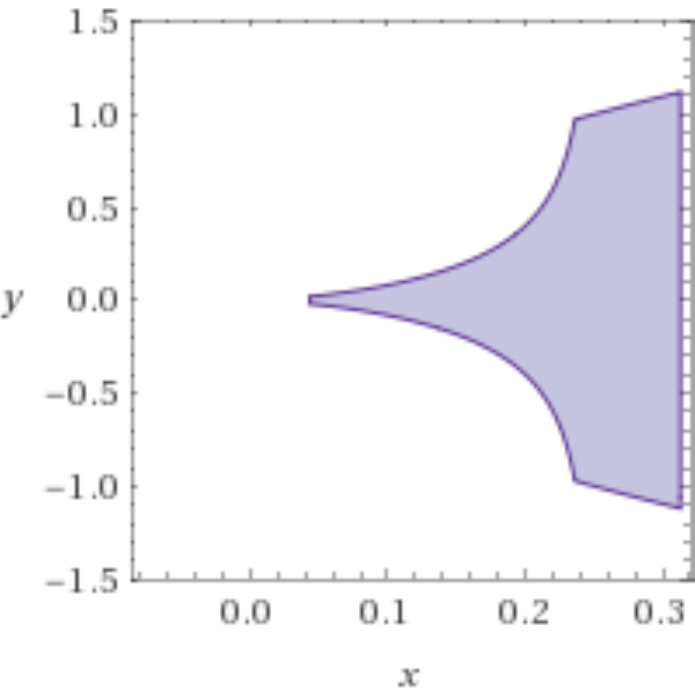}

By Proposition \ref{p3}, this is also the necessary and sufficient condition for the segment $[a,b]\subset \Bbb C^*$ to be in $A_{\pi/2}^2$. 
\end{example} 

\begin{example}
Consider now a 3-element set $S=\lbrace{1,a,b\rbrace}$ and let us give a necessary condition for it to be in $A_{\pi/2}^2$. 

\begin{proposition}\label{p1} Assume $a\notin \Bbb R$. 
Then one has 
$$
a_1\frac{(|1+a|-1-a_1)^2}{a_2^2}\le b_1\le a_1\frac{(|1+a|+1+a_1)^2}{a_2^2}, 
$$
where $a=a_1+ia_2$, $b=b_1+ib_2$ and $a_1,a_2,b_1,b_2\in \Bbb R$. In other words, 
one has $K^{-1}\le \frac{b_1}{a_1}\le K$, where 
$K:=\frac{(|1+a|+1+a_1)^2}{a_2^2}$. 
Thus any $S\in A_{\pi/2}^2$ which is not contained in a line is bounded and separated from the origin. 
\end{proposition}  

\begin{proof} 
We have the inequalities $F(a,1,1,b)\le 0$ and $F(a,b,1,1)\le 0$, which yields
$$
(a_1b_2-a_2b_1)^2\le 4a_1b_1,\quad (a_2-b_2)^2\le 4a_1b_1, 
$$
or, equivalently, 
\begin{equation}\label{twoineq}
|a_1b_2-a_2b_1|\le 2\sqrt{a_1b_1},\quad |a_2-b_2|\le 2\sqrt{a_1b_1}. 
\end{equation}
From the second inequality in \ref{twoineq} we have 
\begin{equation}\label{b2bound}
|b_2|\le 2\sqrt{a_1b_1}+|a_2|.
\end{equation} 
Hence 
$$
|a_1b_2|\le (2\sqrt{a_1b_1}+|a_2|)a_1.
$$
Thus
$$
|a_2b_1|\le 2\sqrt{a_1b_1}+|a_1b_2|\le 2(1+a_1)\sqrt{a_1b_1}+|a_2|a_1. 
$$
Hence 
$$
b_1\le \frac{2(1+a_1)\sqrt{a_1b_1}}{|a_2|}+a_1
$$
This yields 
$$
b_1\le a_1\frac{(|1+a|+1+a_1)^2}{a_2^2},
$$
as claimed. From this we also have 
$$
|a_2|\le 2\sqrt{a_1b_1}+|b_2|\le 2\sqrt{a_1b_1}+\frac{2\sqrt{a_1b_1}+|a_2|b_1}{a_1},
$$
which yields 
$$
b_1\ge \frac{a_1(|1+a|-1-a_1)^2}{a_2^2},
$$
again as claimed. Now \eqref{b2bound} implies that $S$ is bounded if it is not contained in a line.  
\end{proof}

\end{example} 

\subsection{Rectangular and trapezoidal regions} 
Let us now try to characterize rectangular and trapezoidal regions which are in $A_{\pi/2}^2$ (hence in $A_{\pi/2}$). 

\begin{proposition} (i) Let $L,M,N>0$ and $R(M,L,N)$ be the rectangle $M\le x\le M+L$, $|y|\le N$. 
Then $R(M,L,N)\in A_{\pi/2}^2$ if
$$
N\le \frac{2M^{3/2}}{\sqrt{L+24M}}. 
$$
 
(ii) Let $0<M<L$ and $T(M,L,t)$ be the trapezoid 
$$
M\le x\le L, |y|\le tx.
$$ 
Then $T(M,L,t)\in A_{\pi/2}^2$ if $t<\sqrt{2}-1$ and 
$$
L\le M\left(\frac{t^2+t^{-2}-4+(t^{-1}-t)\sqrt{t^2+t^{-2}-6}}{2}\right)^{1/2}
$$
$$
=Mt^{-1}(1+o(t))\text{ as }t\to 0. 
$$
\end{proposition}

\begin{proof} In coordinates the desired basic inequality $F(a,b,c,d)\le 0$ looks like 
$$
(a_2d_1-a_1d_2-b_2c_1+b_1c_2)^2\le 4(a_1c_1+a_2c_2)(b_1d_1+b_2d_2), 
$$ 
where the subscript $1$ denotes the real part and the subscript $2$ the imaginary part
(i.e., $a_1={\rm Re}(a)$, $a_2={\rm Im}(a)$ etc.). 

(i) Since the absolute values of $a_2,b_2,c_2,d_2$ don't exceed $N$, 
the basic inequality would follow from the inequality 
$$
N^2(a_1+b_1+c_1+d_1)^2\le 4(a_1c_1-N^2)(b_1d_1-N^2)=
$$
$$
=4a_1c_1b_1d_1-4N^2(a_1c_1+b_1d_1)+N^4.
$$
(as long as $N\le M$, which follows from the inequality in (i)). This, in turn, would follow from the inequality
$$
N^2((a_1+b_1+c_1+d_1)^2+4(a_1c_1+b_1d_1))\le 4a_1c_1b_1d_1.
$$
Let $q$ be the largest of $a_1,b_1,c_1,d_1$ and $p$ the second largest. 
Then the latter inequality would follow from the inequality 
$$
N^2((a_1+b_1+c_1+d_1)^2+4(a_1c_1+b_1d_1))\le 4M^2pq.
$$
Now observe that on the left hand side we have $24$ quadratic monomials in 
$a_1,b_1,c_1,d_1$, which are all $\le pq$ except one, which is $q^2\le (M+L)q$. 
So the last inequality would follow from the inequality 
 $$
 N^2(23p+M+L)\le 4M^2p, 
 $$
 or 
 $$
N^2(M+L)\le p(4M^2-23N^2). 
$$
This, in turn, follows from the inequality
 $$
N^2(M+L)\le M(4M^2-23N^2), 
$$
or 
$$
N^2(L+24M)\le 4M^3,
$$
giving 
$$
N\le \frac{2M^{3/2}}{\sqrt{L+24M}},
$$
as claimed.  

(ii) Since $|a_2|\le ta_1$, $|b_2|\le tb_1$, $|c_2|\le tc_1$, $|d_2|\le td_1$, the basic inequality would follow from the inequality 
$$
4t^2(a_1d_1+b_1c_1)^2\le 4(1-t^2)^2a_1c_1b_1d_1,
$$
which is equivalent to the inequality
$$
t^2(a_1^2d_1^2+b_1^2c_1^2)\le (1-4t^2+t^4)a_1b_1c_1d_1, 
$$
or $\mu+\frac{1}{\mu}\le t^{-2}-4+t^2$, where $\mu=\frac{a_1d_1}{b_1c_1}$. The largest value of this ratio is $L^2/M^2$, so it sufficient to require that
$$
\frac{L^2}{M^2}+\frac{M^2}{L^2}\le t^2-4+t^{-2}:=T. 
$$
This is satisfied whenever
$$
L\le M\left(\frac{T+\sqrt{T^2-4}}{2}\right)^{1/2},
$$
as claimed. 
\end{proof} 

In particular, if $L=1$ and $M$ is small then for the rectangle we have $N=2M^{3/2}(1+o(M))$. Comparing this to the 
bound \eqref{3/2bd}, we see that this is sharp up to a factor $1+o(M)$. This also relaxes the bound $N\le CM^2$ from \cite{B}. 

Also for the trapezoid we have $M\ge t(1+o(t))$, so its short side has half-length $N=tM$, so the largest possible $N$ is $\sim M^2$. 

\subsection{Maximal angle-restricted sets.} 

From now on we will only consider closed convex sets $S$, since we have seen in Proposition \ref{p3} that if $S\in A_{\pi/2}^2$ then so do its closure and its convex hull, and a convex set is in $A_{\pi/2}$ iff it is in $A_{\pi/2}^2$.

It is clear from Zorn's lemma that any $(\pi/2,\pi/2)$-angle restricted set is contained in a maximal one, which  is necessarily closed and convex. The problem of finding and classifying maximal $(\pi/2,\pi/2)$-angle-restricted sets is a special case of a more general problem of optimal control theory -- to find maximal regions $R$ with the property that a given function $F(z_1,...,z_n)$ is $\le 0$ when all $z_i\in R$; one of the simplest and best known problems from this family is to describe curves of constant width $\ell$ (in this case $F(z_1,z_2)=|z_1-z_2|^2-\ell^2$). As is typical for such problems, the problem of describing maximal regions in $A_{\pi/2}$ is rather nontrivial; presumably, it can be treated by the methods of the book \cite{BCGGG}. 

Maximal regions can also be constructed as limits of nested sequences $\Pi_n$ of convex $n$-gons, each obtained from the previous one by ``pushing out" a point on one of the sides as far as it can go while still preserving the property of being in $A_{\pi/2}$. This approach should be good for numerical computation of maximal regions, since the verification that the region is in $A_{\pi/2}^2$ (equivalently, in $A_{\pi/2}$) is just a finite check on the vertices of the polygon. 

Here we will not delve into this theory and will restrict ourselves to proving the following result. Let 
$\mu_S(a):={\rm max}_{b,c,d\in S}F(a,b,c,d)$. 
We have seen that $S\in A_{\pi/2}$ iff $\mu_S(a)\le 0\ \forall a\in S$. 

\begin{proposition} A closed convex set $S\in A_{\pi/2}$ (not contained in a line) 
is maximal iff $\mu_S(a)=0$ for all $a\in \partial S$. 
\end{proposition} 

\begin{proof} Note that $S$ is bounded by Proposition \ref{p1}, hence compact. 
Suppose $S\in A_{\pi/2}$ is maximal and $a\in \partial S$ is such that there are no $b,c,d\in S$ with $F(a,b,c,d)=0$. 
Then $\mu_S(a)=-\varepsilon<0$. Now take sufficiently small $\delta$ and let $S'=S\cup \lbrace{|z-a|\le\delta\rbrace}$, which is strictly larger than $S$ as 
$a\in \partial S$. Let us maximize $F(x,b,c,d)$ over $x,b,c,d\in S'$. If these points are further than $\delta$ from $a$ then 
they are in $S$ so $F(x,b,c,d)\le 0$. Otherwise, if one of them is $\delta$-close to $a$, say, $x$ (it does not matter which one because of the permutation symmetry of $F$), then $F(x,b,c,d)\le F(a,b,c,d)+\varepsilon\le 0$ (a number $\delta$ with this property exists due to uniform continuity of $F$ on $S$). So 
$S'$ and its convex hull are in $A_{\pi/2}$, contradicting the assumption that $S$ is maximal. 

Conversely, suppose $\mu_S(a)=0$ on $\partial S$, let $S'\supset S$ be a larger convex region. Then there exists $a\in \partial S$ which is an interior point of $S'$. Also there exist $b,c,d\in S$ with $F(a,b,c,d)=0$. But for fixed $b,c,d$ the function $F(z,b,c,d)$ is inhomogeneous quadratic in $z,\bar z$ with nonnegative degree $2$ part, which implies that there is a point $a'$ arbitrarily close to $a$ with $F(a',b,c,d)>0$. Hence $S'\notin A_{\pi/2}$ and $S$ is maximal. 
\end{proof} 

Thus, we see that if $S\in A_{\pi/2}$ and $a\in \partial S$ with $\mu_S(a)<0$ then $S$ can be enlarged near $a$ (e.g. by adding a point $a'\notin S$ close to $a$ and taking the convex hull of $S$ and $a'$), so that the larger set $S'$ is still in $A_{\pi/2}$. Otherwise, if $\mu_S(a)=0$, then $a$ must be on the boundary of any $S'\in A_{\pi/2}$ containing $S$. We will say that $S$ is maximal at $a$ if $\mu_S(a)=0$ and non-maximal at $a$ if $\mu_S(a)<0$. 

\begin{example} Let $S$ be the disk $|z-1|\le 1/2$. Then it is easy to check that $S$ is maximal at the three points $a=1/2,1\pm i/2$ (indeed, picking $b,c,d$ from the same set, we can make $F(a,b,c,d)=0$). On the other hand, we claim that $S$ is 
{\it not} maximal at any other points of the boundary circle. The proof is by a direct computation. Namely, if $a\ne 1/2,1\pm i/2$ but $|a-1|=1/2$, then it can be shown that for any $b$ with $|b-1|\le 1/2$ one has $G_1(a,b)<0$ and $G_2(a,b)<0$ (hence, any small perturbation of $S$ at $a$ will still be in $B_{\pi/2}^2$, hence in $A_{\pi/2}^2$). 
Indeed, setting $a=1+\frac{1}{2}e^{i(u-\pi/4)}$ and $b=1+\frac{1}{2}e^{i(v-\pi/4)}$ for $u,v\in \Bbb R/2\pi \Bbb Z$, we have 
$$
G_1(a,b)=H(u,v):=\frac{1}{4}(\sin u-\sin v)^2-(\sqrt{2}+\cos u)(\sqrt{2}+\cos v),
$$
and maximization of this function (e.g., using Wolfram Alpha, or analytically) yields $H(u,v)\le 0$, and $H(u,v)=0$ if and only if $u=-3\pi/4$ and $v=3\pi/4$ or 
$u=3\pi/4$ and $v=-3\pi/4$, which implies the desired statement. 
\end{example}


\begin{thebibliography}{9999999}
\bibitem[B1]{B} A. Barvinok, Approximating permanents and Hafnians, 
Discrete analysis, 2017:2, 34 pp.

\bibitem[B2]{B2} A. Barvinok, Combinatorics and complexity of partition functions, 
Springer, 2016. 

\bibitem[BCGGG]{BCGGG} R.L. Bryant, S.S. Chern, R.B. Gardner, H.L. Goldschmidt, and P.A. Griffiths, Exterior Differential Systems. Springer, 2011. 

\bibitem[JSV]{JSV} M. Jerrum, A. Sinclair, E. Vigoda, A Polynomial-Time Approximation Algorithm for the Permanent of a Matrix with Nonnegative Entries, Journal of the ACM, Volume 51, Issue 4, July 2004,
Pages 671-697. 
\end{thebibliography}
\end{document}